\documentclass[11pt,twoside,leqno]{article}

\usepackage{graphics}
\usepackage{graphicx}
\usepackage{amssymb}
\usepackage{amsmath}
\usepackage{amsthm}
\usepackage{color}
\usepackage{mathrsfs}

\usepackage{indentfirst}

\usepackage{txfonts}

\allowdisplaybreaks

\def\red{\color{red}}

\def\rr{{\mathbb R}}
\def\rn{{\mathbb{R}^n}}

\def\cc{{\mathbb C}}

\def\nn{{\mathbb N}}

\def\cf{{\mathcal F}}

\def\cm{{\mathcal M}}

\def\fz{\infty }
\def\az{\alpha}

\def\gz{{\gamma}}

\def\oz{{\omega}}

\def\tz{\theta}

\def\lf{\left}
\def\r{\right}

\def\ls{\lesssim}

\def\noz{\nonumber}

\def\dist{\mathop\mathrm{\,dist\,}}

\def\BMO{\mathop\mathrm{\,BMO\,}}

\def\CMO{\mathop\mathrm{\,CMO\,}}
\def\BMOa{\mathop\mathrm{\,BMO_\alpha\,}}
\def\CMOa{\mathop\mathrm{\,CMO_\alpha\,}}

\def\loc{{\mathop\mathrm{\,loc\,}}}
\def\supp{\mathop\mathrm{\,supp\,}}

\def\XXint#1#2#3{{\setbox0=\hbox{$#1{#2#3}{\int}$ }
\vcenter{\hbox{$#2#3$ }}\kern-.6\wd0}}

\def\({\left(}
\def \){ \right)}
 \def\ve{{\varepsilon}}

\newtheorem{theorem}{Theorem}[section]
\newtheorem{lemma}[theorem]{Lemma}
\newtheorem{corollary}[theorem]{Corollary}

\newtheorem{proposition}[theorem]{Proposition}

\theoremstyle{definition}
\newtheorem{remark}[theorem]{Remark}

\renewcommand{\appendix}{\par
   \setcounter{section}{0}%
   \setcounter{subsection}{0}%
   \setcounter{subsubsection}{0}%
   \gdef\thesection{\@Alph\c@section}%
   \gdef\thesubsection{\@Alph\c@section.\@arabic\c@subsection}%
   \gdef\theHsection{\@Alph\c@section.}%
   \gdef\theHsubsection{\@Alph\c@section.\@arabic\c@subsection}%
   \csname appendixmore\endcsname
 }

\numberwithin{equation}{section}

\begin{document}

\title{\bf\Large The $L^p$-to-$L^q$ Compactness of Commutators with $p>q$
\footnotetext{\hspace{-0.35cm} 2020 {\it
Mathematics Subject Classification}. Primary 47B47;
Secondary 42B20, 42B25, 42B35, 46E30.
% 47B47 (1973-now) Commutators, derivations, elementary operators, etc.
% 42B20 (1980-now) Singular and oscillatory integrals (Calder\'on--Zygmund, etc.)
% 42B25 (1980-now) Maximal functions, Littlewood-Paley theory
% 42B35 (2000-now) Function spaces arising in harmonic analysis
% 46E30 (1973-now) Spaces of measurable functions (Lp-spaces, Orlicz spaces, K\"othe function spaces, Lorentz spaces, rearrangement invariant spaces, ideal spaces, etc.)
\endgraf {\it Key words and phrases.} Commutator, compactness,
Calder\'on--Zygmund operator,
Fr\'{e}chet--Kolmogorov theorem.
\endgraf Tuomas Hyt\"onen is supported by
the Academy of Finland (Grant Nos.\ 314829 and 346314). Kangwei Li is supported by the National Natural Science Foundation of China (Grant No. 12001400).
Jin Tao and Dachun Yang are supported by
the National Key Research and Development Program of China
(Grant No. 2020YFA0712900) and
the National Natural Science Foundation of China
(Grant Nos. 11971058 and 12071197).}}
\author{Tuomas Hyt\"onen, Kangwei Li, Jin Tao and Dachun Yang\footnote{Corresponding author,
E-mail: \texttt{dcyang@bnu.edu.cn}/{\red August 22, 2022}/Final version.}}
\date{}
\maketitle

\vspace{-0.7cm}

\begin{center}
\begin{minipage}{13cm}
{\small {\bf Abstract}\quad
Let $1<q<p<\fz$, $\frac1r:=\frac1q-\frac1p$,
and $T$ be a non-degenerate Calder\'on--Zygmund operator.
We show that the commutator $[b,T]$
is compact from $L^p({\mathbb R}^n)$ to
$L^q({\mathbb R}^n)$ if and only if the symbol
$b=a+c$ with $a\in L^r({\mathbb R}^n)$ and $c$ being any constant.
Since both the corresponding Hardy--Littlewood maximal operator
and the corresponding Calder\'on--Zygmund maximal operator
are not bounded from $L^p({\mathbb R}^n)$ to $L^q({\mathbb R}^n)$,
we take the full advantage of the compact support of
the approximation element in $C_{\rm c}^\infty({\mathbb R}^n)$,
which seems to be redundant for many corresponding estimates when $p\leq q$
but to be crucial when $p>q$. We also extend the results to the multilinear case.}
\end{minipage}
\end{center}

\vspace{0.1cm}

%\tableofcontents

\vspace{0.1cm}

\section{Introduction\label{s1}}

Let $T$ denote a Calder\'on--Zygmund operator
and $\BMO(\rn)$ be the space of functions with bounded mean oscillation
introduced by John and Nirenberg \cite{jn61}.
In the celebrated article of Coifman et al. \cite{crw76},
they proved that the commutator
$$[b,T](f):=bT(f)-T(bf)$$
is bounded on $L^p(\rn)$ with $p\in(1,\fz)$ if and only if
the symbol $b$ belongs to $\BMO(\rn)$.
Later, Uchiyama \cite{u78} showed that
$[b,T]$ is compact on $L^p(\rn)$ with $p\in(1,\fz)$ if and only if
the symbol $b$ belongs to $\CMO(\rn)$, namely,
the closure of $C_{\rm c}^\fz(\rn)$ in $\BMO(\rn)$.
Here and hereafter, $C_{\rm c}^\fz(\rn)$ denotes the set of all
infinitely differentiable functions on $\rn$ with compact support.
Since then, the commutators, generated by Calder\'on--Zygmund operators
and $\BMO(\rn)$ functions, have been
widely studied in various branches of mathematics;
see, for instance, \cite{cdlw19,cllv22,dgklwy21,dllw19,ns17cm,ns17,sy18}.

Similar characterizations of the boundedness and the compactness
also hold true in the fractional setting.
To be precise, let $1<p<q<\fz$ and $b$ belong to the fractional
variant of $\BMO(\rn)$, that is,
$$\|b\|_{\BMOa(\rn)}:=\sup_{\mathrm{cube}\ Q\subset\rn}
\frac1{|Q|^{1+\az/n}}\int_Q\lf|b(x)-b_Q\r|\,dx<\fz$$
with $\az\in(0,1)$ and
$$b_Q:=\frac1{|Q|}\int_Q b(y)\,dy.$$
Then $\BMOa(\rn)$ when $\az=0$ coincides with $\BMO(\rn)$,
and one can derive the boundedness of commutators $[b,T]$
from $L^p(\rn)$ to $L^q(\rn)$ with the help of
fractional integrals; see, for instance, \cite[Theorem 1.0.1]{h21}.
The corresponding characterization of the
compactness also holds true, for any $\az\in[0,1)$, via choosing $b$ in
the closure of $C_{\rm c}^\fz(\rn)$ in $\BMOa(\rn)$; see, for instance, \cite[Theorem 1.8]{ghwy22}.
The case $\az=1$ requires a different formulation in terms of
a suitably defined space $\CMOa(\rn)$
which agrees with the aforementioned closure for $\az\in[0,1)$,
but not for $\az=1$; see \cite[Theorems 1.7 and 1.8]{ghwy22}.
When considering the boundedness and the compactness of commutators,
we always use some classical operators, such as (fractional) maximal operators
and fractional integrals, which map $L^p(\rn)$ to $L^q(\rn)$ with $p\le q$.
However, for the case $p>q$, these operators are no longer bounded.
This may be a possible reason why the range $p>q$ has attracted little attention for a long time.

Recently, the first author \cite{h21} of this article completely settled
the case $p>q$ on the boundedness of commutators $[b,T]$
from $L^p(\rn)$ to $L^q(\rn)$, and hence clarified
the boundedness of commutators for all $p,q\in(1,\fz)$.
This surprising characterization says that,
for any ``non-degenerate'' Calder\'on--Zygmund operator $T$,
the commutator $[b,T]$ is bounded from $L^p(\rn)$ to $L^q(\rn)$
with $1<q<p<\fz$ if and only if $b\in L^r(\rn)$ modulo constants,
where $\frac1r:=\frac1q-\frac1p$.
As an application, following an approach proposed by Lindberg \cite{Lind17},
the first author \cite{h21} of this article further showed
that every $f\in L^p(\rn)$ can be represented as a convergent series of
normalized Jacobians $Ju=\nabla u$ of $u\in \dot W^{1,np}(\rn)^n$,
which extends a famous result of Coifman et al. \cite{clms93}
and supports a conjecture of Iwaniec \cite{i97}
about the solvability of the prescribed Jacobian equation $Ju=f\in L^p(\rn)$.

In another direction, a quite general approach to the compactness of linear operators
on the weighted space $L^p_\omega(\rn)$ has been established very recently
by the first author of this article and Lappas \cite{hl22RMI,hl22IMNS},
There are also recent results about the compactness of Calder\'on--Zygmund 
commutators in the two weight setting,
namely, from $L^p_\sigma(\rn)$ to $L^q_\omega(\rn)$ with different weights; we refer the reader to
Lacey and Li \cite{ll22} for case $p=q$ and to Oikari, Sinko and one of us \cite{HOS} for $p<q$.
But the compactness of $[b,T]$ from $L^p(\rn)$ to $L^q(\rn)$ with $p>q$
is still unknown so far, even in the unweighted case.

The main purpose of this article is to investigate the
compactness of commutators $[b,T]$ from $L^p(\rn)$ to $L^q(\rn)$
for any given $1<q<p<\fz$.
Besides natural interest in its own right, this question is motivated 
by the needs of Lindberg's program to the mentioned Jacobian 
conjecture of Iwaniec, as further developed in \cite{Lind22}.
Indeed, we show in Section \ref{s2} that
the boundedness of $[b,T]$ is \emph{equivalent} to
the compactness of $[b,T]$ in this range;
see Theorem \ref{thm-cpt-char} below.
This equivalence is essentially based on the density of
$C_{\rm c}^\fz(\rn)$ in $L^r(\rn)$.
Since the implication ``compact $\Longrightarrow$ bounded'' is obvious,
the main contribution of this article is
the reverse implication ``bounded $\Longrightarrow$ compact''
obtained in Theorems \ref{thm-cpt} and \ref{thm=cpt-Toz} below.
To this end, we apply the criterion of the compactness in Lebesgue spaces,
namely, the Fr\'{e}chet--Kolmogorov theorem; see Lemma \ref{lem-FK} below.
Since both the corresponding Hardy--Littlewood maximal operator
and the corresponding Calder\'on--Zygmund maximal operator
are not bounded from $L^p({\mathbb R}^n)$ to $L^q({\mathbb R}^n)$,
we take the full advantage of the compact support of
the approximating element in $C_{\rm c}^\infty({\mathbb R}^n)$,
which seems to be redundant for many corresponding estimates when $p\leq q$
but to be crucial when $p>q$.
Moreover, we also discuss the corresponding multilinear and iterated results
in Section \ref{s3}. Apart from using the median method
and the expectation of random signs,
we also apply a characterization of $\dot L^r(\rn)$
via sparse collections of cubes; see Proposition \ref{prop:lr} below.

Throughout this article, we denote by $C$ and $\widetilde{C}$
positive constants that are independent of the main parameters
under consideration, but they may vary from line to line.
Moreover, we use $C_{(\gamma,\ \beta,\ \ldots)}$ to denote
a positive constant depending on the indicated
parameters $\gamma,\ \beta,\ \ldots$.
Constants with subscripts, such as $C_{0}$ and $A_1$,
do not change in different occurrences.
Moreover, the {symbol} $f\lesssim g$ represents that
$f\le Cg$ for some positive constant $C$.
If $f\lesssim g$ and $g\lesssim f$, we then write $f\sim g$.
If $f\le Cg$ and $g=h$ or $g\le h$, we then write $f\ls g\sim h$
or $f\ls g\ls h$, rather than $f\ls g=h$ or $f\ls g\le h$.
For any $p\in[1,\fz]$, let $p'$ denote its \emph{conjugate index},
that is, $p'$ satisfies  $1/p+1/p'=1$.
We use $\mathbf{0}$ to denote the origin of $\rn$ and,
for any set $E\subset \rn$, we use $\mathbf{1}_E$ to denote its
characteristic function.
Also, for any $\gz\in(0,\fz]$, the \emph{Lebesgue space}
$L^\gz(\rn)$ is defined to be the set of all the Lebesgue measurable
functions $f$ on $\rn$ such that
$$\|f\|_\gz:=\left[\int_{\rn}|f(x)|^\gz\,dx\right]^{1/\gz}$$
with usual modification made when $\gz=\fz$.
For any $r\in(0,\fz]$, we use $L^r_\loc(\rn)$ to denote the set
of all the Lebesgue measurable functions $f$ on $\rn$ such that
$\|f\mathbf{1}_B\|_r<\fz$ for any ball $B\subset\rn$.
Furthermore, for any $r\in(0,\fz]$, let
\[
\dot L^r(\rn):=\lf\{b\in L^r_\loc(\rn):\
\| b\|_{\dot L^r(\rn)}=\inf_c \| b-c\|_{r}<\infty\r\};
\]
for any $b\in L^1_\loc(\rn)$ and any bounded set $S\subset\rn$ with $|S|>0$,
let
$$\langle b \rangle_S:=\frac{1}{|S|}\int_S b(x)\,dx.$$

\section{Main results and proofs\label{s2}}
To obtain the compactness,
we apply the following Fr\'{e}chet--Kolmogorov theorem
which can be found in, for instance, \cite[p.\,275]{y95} and \cite{T51};
see also \cite[Theorem 1.1]{XYY}.
\begin{lemma}[Fr\'{e}chet--Kolmogorov]\label{lem-FK}
Let $q \in (0,\infty)$.
Then $\cf\subset L^q(\rn)$ is relatively compact in $L^q(\rn)$ if and only if
\begin{enumerate}
\item[\rm(i)] $\cf$ is bounded, namely,
$$\sup_{f\in \cf} \|f\|_{q}<\infty;$$
\item[\rm(ii)] $\cf$ uniformly vanishes at infinity, namely,
$$\lim_{M\to\infty}\lf\|f{\mathbf 1}_{\{x\in\rr:\ |x|\ge M\}}\r\|_{q}=0\quad
\text{uniformly\,\,for\,\,all}\,\, f\in \cf.$$
\item[\rm(iii)] $\cf$ is equicontinuous, namely,
$$\lim_{\xi\to\mathbf{0}}\|f(\cdot +\xi)-f(\cdot)\|_{q}=0\quad
\text{uniformly\,\,for\,\,all}\,\, f\in \cf.$$
\end{enumerate}
\end{lemma}

In what follows, we consider a Calder\'on--Zygmund kernel $K$
under the standard conditions
\begin{align}\label{size}
|K(x,y)|\le \frac{C_{(K)}}{|x-y|^n}\quad{\rm for\ any}\
x,y\in\rn\ {\rm and}\ x\neq y,
\end{align}
and, for any $x,\widetilde{x},y\in\rn$ with $|x-\widetilde{x}|<\frac12 |x-y|$,
\begin{align}\label{regular}
|K(x,y)-K(\widetilde{x},y)|+|K(y,x)-K(y,\widetilde{x})|
\le \frac{C_{(K)}}{|x-y|^n}\oz\lf(\frac{|x-\widetilde{x}|}{|x-y|}\r),
\end{align}
where $C_{(K)}$ denotes some positive constant depending only on $K$
and the modulus of continuity $\oz:\ [0,1)\to[0,\fz)$ is increasing.
Moreover, we say that $\oz$ satisfies the \emph{Dini condition} if
\begin{align}\label{Dini}
\int_{0}^1 \frac{\oz(t)}{t}\,dt<\fz.
\end{align}
Notice that $\oz(t):=t^\az$ for any $t\in[0,\fz)$ and
some $\az\in(0,1]$ satisfies the Dini condition.
As usual, we assume that the Calder\'on--Zygmund operator $T$
is bounded on $L^p(\rn)$ for any $p\in(1,\fz)$.

We also need the boundedness of the following two maximal operators on $L^p(\rn)$.
Recall that the \emph{Hardy--Littlewood maximal operator}
$\cm$ is defined by setting,
for any $f\in L^1_\loc(\rn)$ and $x\in\rn$,
$$\cm(f)(x):=\sup_{{\rm ball}\ B\ni x}\frac1{|B|}\int_B|f(y)|\,dy;$$
and the \emph{maximal truncated operator} $T^\ast$ is defined by setting,
for any $f\in C_{\rm c}^\fz(\rn)$ and $x\in\rn$,
$$T^\ast(f)(x):=\sup_{\zeta\in(0,\fz)}\lf|\int_{\{y\in\rn:\ |x-y|>\zeta\}}
K(x,y)f(y)\,dy\r|,$$
where the kernel $K$ satisfies \eqref{size}, \eqref{regular}, and \eqref{Dini}.
The following boundedness of both $\cm$ and $T^\ast$ on $L^p(\rn)$ are well known;
see, for instance, \cite[Theorems 2.5 and 5.14]{d01}.
\begin{lemma}\label{lem-max-bdd}
Let $p\in(1,\fz)$. Then both $\cm$ and $T^\ast$ are bounded on $L^p(\rn)$.
\end{lemma}

To use the density argument of compact operators,
we adapt the smooth truncated technique same as in \cite{kl01,cc13,tyyz21PA}.
Let $\eta\in(0,\fz)$ and $T_\eta$ be the smooth truncated Calder\'on--Zygmund operator
of $T$, that is, $T_\eta$ is generated by the kernel
\begin{align}\label{Keta-def}
K_\eta(x,y):=K(x,y)\lf[1-\varphi\lf(\frac{|x-y|}{\eta}\r)\r]\quad
\text{for any }x,y\in\rn
\end{align}
with $\varphi\in C^\fz([0,\fz))$ satisfying both
\begin{align}\label{phi-def}
0\le\varphi\le1{\quad \rm and\quad }
\varphi(x)=
\begin{cases}
1, &x\in[0,1/2],\\
0, &x\in[1,\fz].
\end{cases}
\end{align}
In what follows, for any $p,q\in(1,\fz)$,
we use $\|T\|_{p\to q}$ to denote the \emph{operator norm} of $T$
from $L^p(\rn)$ to $L^q(\rn)$.
For any $p\in[1,\infty]$, we denote by $p'$ its \emph{conjugate exponent},
namely, $1/p+1/p'=1$.
In addition, $C_{\rm c}^1(\rn)$ denotes the set of all
the differentiable functions $b$ on $\rn$ such that
$b$ has compact support and its gradient $\nabla b$ is continuous.
Then we have the following approximation.
In what follows, the symbol $\eta\to0^+$ means $\eta\in(0,\fz)$ and $\eta\to0$.
\begin{lemma}\label{lem-Teta}
Let $1<q<p<\fz$, $b\in C_{\rm c}^1(\rn)$,
and $T$ be a Calder\'on--Zygmund operator whose kernel
satisfies \eqref{size}, \eqref{regular}, and \eqref{Dini}.
Then, for any $\eta\in(0,\fz)$,
$[b,T]-[b,T_\eta]$ is bounded from $L^p(\rn)$ to $L^q(\rn)$ and
$$\lim_{\eta\to0^+}\lf\|[b,T]-[b,T_\eta]\r\|_{p\to q}=0.$$
\end{lemma}

\begin{proof}
Let all the symbols be the same as in the present lemma
and let $\eta\in(0,\fz)$.
Then $\supp(b)\subset B(\mathbf{0},R_0)$ for some positive constant $R_0$,
and hence, for any  $y\in\rn\setminus B(\mathbf{0},R_0+\eta)$
and any $x\in\rn$ with $|x-y|\le\eta$, we have $b(x)=0=b(y)$ and hence
\begin{align}\label{bx-by=0}
b(x)-b(y)=0.
\end{align}
In what follows, without loss of generality, we may consider that
both $T$ and $T_\eta$ for any $\eta\in(0,\fz)$ are well defined on
any $f\in L^p(\rn)$;
otherwise, we need first to consider $f\in C_{\rm c}^\fz(\rn)$
and then to use a density argument.
Thus, by \eqref{Keta-def}, \eqref{phi-def}, \eqref{bx-by=0}, \eqref{size},
and the mean value theorem, we conclude that,
for any $\eta\in(0,\fz)$, $f\in L^p(\rn)$, and $x\in\rn$,
\begin{align*}
&\lf|[b,T](f)(x)- [b,T_\eta](f)(x)\r|\\
&\quad=\lf|\lim_{\varepsilon\to0^+}\int_{\varepsilon<|x-y|<1/\varepsilon}
[b(x)-b(y)]K(x,y)
\varphi\lf(\frac{|x-y|}{\eta}\r)f(y)\,dy \r|\\
&\quad\le\int_{\rn}|b(x)-b(y)|\lf|K(x,y)\r|\mathbf{1}_{\{y\in\rn:\ |x-y|\le\eta\}}(x,y)
\lf|f(y)\r|\mathbf{1}_{B(\mathbf{0},R_0+\eta)}(y)\,dy\\
&\quad\ls\int_{\rn}|b(x)-b(y)|
\frac{\mathbf{1}_{\{y\in\rn:\ |x-y|\le\eta\}}(x,y)}{|x-y|^n}
|f(y)|\mathbf{1}_{B(\mathbf{0},R_0+\eta)}(y)\,dy\\
&\quad\ls\|\nabla b\|_{\fz}
\sum_{k=0}^\fz\int_{\{y\in\rn:\ 2^{-(k+1)}\eta<|x-y|\le2^{-k}\eta\}}
\frac{|f(y)|\mathbf{1}_{B(\mathbf{0},R_0+\eta)}(y)}{|x-y|^{n-1}}\,dy\\
&\quad\ls\|\nabla b\|_{\fz}
\sum_{k=0}^\fz \frac{2^{-k}\eta}{[2^{-(k+1)}\eta]^n}
\int_{\{y\in\rn:\ |x-y|\le2^{-k}\eta\}}
|f(y)|\mathbf{1}_{B(\mathbf{0},R_0+\eta)}(y)\,dy\\
&\quad\ls\eta \|\nabla b\|_{\fz} \cm(|f|\mathbf{1}_{B(\mathbf{0},R_0+\eta)})(x),
\end{align*}
where $\|\nabla b\|_{\fz}$ denotes the essential supremum of $|\nabla b|$ on $\rn$
and $\cm$ the Hardy--Littlewood maximal operator.
From this, Lemma \ref{lem-max-bdd}, and the H\"older inequality,
it follows that
\begin{align*}
&\lf\|[b,T](f)- [b,T_\eta](f)\r\|_q\\
&\quad\ls \eta \|\nabla b\|_{\fz} \lf\|\cm(|f|\mathbf{1}_{B(\mathbf{0},R_0+\eta)}) \r\|_q
\ls \eta \|\nabla b\|_{\fz} \|\cm\|_{q\to q} \lf\||f|\mathbf{1}_{B(\mathbf{0},R_0+\eta)} \r\|_q\\
&\quad\ls \eta \|\nabla b\|_{\fz} \|\cm\|_{q\to q} \lf\|\mathbf{1}_{B(\mathbf{0},R_0+\eta)} \r\|_r
\|f\|_p\\
&\quad\sim \eta \|\nabla b\|_{\fz} \|\cm\|_{q\to q} \lf(R_0+\eta\r)^{n/r} \|f\|_p
\end{align*}
and hence $[b,T]-[b,T_\eta]$ is bounded from $L^p(\rn)$ to $L^q(\rn)$ and
$$\lim_{\eta\to0^+}\lf\|[b,T]-[b,T_\eta]\r\|_{p\to q}=0.$$
This finishes the proof of Lemma \ref{lem-Teta}.
\end{proof}

\begin{remark}\label{rem-approx}
\begin{enumerate}
\item[\rm(i)] Let all the symbols be the same as in Lemma \ref{lem-Teta}.
Then, by \cite[Theorem 1.0.1]{h21} and
$b\in C_{\rm c}^1(\rn) \subset L^r(\rn)$ with $\frac1r:=\frac1p-\frac1q$,
we know that $[b,T]$ is bounded from $L^p(\rn)$ to $L^q(\rn)$,
which, together with Lemma \ref{lem-Teta},
further implies that $[b,T_\eta]$ is bounded from $L^p(\rn)$ to $L^q(\rn)$
for any $\eta\in(0,\fz)$.

\item[\rm(ii)] Observe that the compact support of $b\in C_{\rm c}^1(\rn)$
plays a key role in the proof of Lemma \ref{lem-Teta}.
In contrast to this, the corresponding approximation from $L^p(\rn)$ to $L^q(\rn)$ with $p\le q$
only needs the smoothness of $b$, namely, $\|\nabla b\|_{L^\fz(\rn)}<\fz$;
see, for instance, \cite[Lemma 7]{cc13}, \cite[Lemma 3.1]{txyy21JFAA},
and \cite[Lemma 3.4]{tyy22MANN}.
\end{enumerate}
\end{remark}

Now, we state the first main theorem of this article.

\begin{theorem}\label{thm-cpt}
Let $1<q<p<\fz$, $b\in L^r(\rn)$ with $\frac1r:=\frac1q-\frac1p$,
and $T$ be a Calder\'on--Zygmund operator whose kernel
satisfies \eqref{size}, \eqref{regular}, and \eqref{Dini}.
Then the commutator $[b,T]$ is compact from $L^p(\rn)$ to $L^q(\rn)$.
\end{theorem}
\begin{proof}
Let all the symbols be the same as in the present theorem.
To show the desired compactness,
by the density of $C_{\rm c}^\fz(\rn)$ in $L^r(\rn)$,
Remark \ref{rem-approx}(i), Lemma \ref{lem-Teta},
and \cite[p.\,278, Theorem(iii)]{y95},
it suffices to prove that,
for any given $b\in C_{\rm c}^\fz(\rn)$ and any $\eta\in(0,\fz)$ small enough,
$[b,T_\eta]$ is compact from $L^p(\rn)$ to $L^q(\rn)$.
To this end, for any bounded subset $\cf\subset L^p(\rn)$, we show that
the set $[b,T]\cf:=\{[b,T](f):\ f\in\cf\}$ satisfies (i), (ii),
and (iii) of Lemma \ref{lem-FK}, and we proceed in order.

First, using both Lemma \ref{lem-Teta} and Remark \ref{rem-approx}(i),
we find that $[b,T_\eta]\mathcal{F}$ satisfies the condition (i) of Lemma \ref{lem-FK}.

Next, from $b\in C_{\rm c}^\fz(\rn)$, it follows that
$\supp (b)\subset B(\mathbf{0},R_0)$ for some positive constant $R_0$.
Let $M\in(2R_0,\infty)$. Then,
for any $y\in B(\mathbf{0},R_0)$ and $x\in\rn$ with $|x|\in(M,\infty)$,
we have $|x-y|\sim |x|$.
Moreover, by this, \eqref{Keta-def}, \eqref{size}, and the H\"older inequality,
we conclude that, for any $f\in\cf$ and
$x\in\rn$ with $|x|\in(M,\infty)$,
\begin{align*}
\lf|[b,\,T_\eta](f)(x)\r|
&\ls\int_\rn
\frac{|b(x)-b(y)|}{|x-y|^n}|f(y)|\,dy
\ls\|b\|_\fz\int_{B(\mathbf{0},R_0)}
\frac{|f(y)|}{|x|^n}\,dy\\
&\ls\frac{\|b\|_\fz}{|x|^n}\|f\|_{p}\lf\|{\mathbf 1}_{B(\mathbf{0},R_0)}\r\|_{p'}
\ls\frac{\|b\|_\fz\|f\|_{p}R_0^{n/p'}}{|x|^n}
\end{align*}
and hence
\begin{align*}
&\lf\|[b,\,T_\eta](f){\mathbf 1}_{\{x\in\rn:\ |x|>M\}}\r\|_{q}\noz\\
&\quad\ls \|b\|_\fz\|f\|_{p}R_0^{n/p'}
\sum_{j=0}^\infty
\lf\|\frac1{|\cdot|^n} {\mathbf 1}_{\{x\in\rn:\ 2^jM<|x|\leq 2^{j+1}M\}}\r\|_{q}\noz\\
&\quad\ls \|b\|_\fz\|f\|_{p}R_0^{n/p'}
\sum_{j=0}^\infty
\frac{\|{\mathbf 1}_{\{x\in\rn:\ |x|\leq 2^{j+1}M\}}\|_{q}}{(2^jM)^{n}}\noz\\
&\quad\ls \|b\|_\fz\|f\|_{p}R_0^{n/p'}
\sum_{j=0}^\infty
\frac{(2^{j+1}M)^{n/q}}{(2^jM)^{n}}
\sim\frac{\|b\|_\fz\|f\|_{p}R_0^{n/p'}}{M^{n/q'}}.
\end{align*}
Therefore, the condition (ii) of Lemma \ref{lem-FK} holds true for
$[b,T_\eta]\mathcal{F}$.

It remains to prove that $[b,T_\Omega^{(\eta)}]\mathcal{F}$
also satisfies the condition (iii) of Lemma \ref{lem-FK}.
For any $f\in\cf$,
$\xi\in\rn\setminus\{\mathbf{0}\}$,
and $x\in\rn$, we have
\begin{align}\label{L1+L2}
&[b, T_\eta](f)(x)-[b, T_\eta](f)(x+\xi)\\
&\quad=\int_{\rn}[b(x)-b(y)] K_\eta(x,y) f(y)\,dy\noz\\
&\qquad-\int_{\rn}[b(x+\xi)-b(y)] K_\eta(x+\xi,y) f(y)\,dy\notag\\
&\quad=[b(x)-b(x+\xi)]\int_{\rn} K_\eta(x,y) f(y)\,dy\notag\\
&\qquad+\int_{\rn}[b(x+\xi)-b(y)] \lf[K_\eta(x,y)-K_\eta(x+\xi,y)\r] f(y)\,dy\notag\\
&\quad=[b(x)-b(x+\xi)]\int_{\rn} K_\eta(x,y) f(y)\,dy\notag\\
&\qquad+\int_{B(\mathbf{0}, R_0)}[b(x+\xi)-b(y)]
\lf[K_\eta(x,y)-K_\eta(x+\xi,y)\r] f(y)\,dy\notag\\
&\qquad+b(x+\xi)\int_{\rn\setminus B(\mathbf{0}, R_0)}
\lf[K_\eta(x,y)-K_\eta(x+\xi,y)\r] f(y)\,dy\notag\\
&\quad=:L_1(x)+L_2(x)+L_3(x).\noz
\end{align}
We first estimate $L_1$. For any $x\in\rn$,
using \eqref{size}, we obtain
\begin{align*}
\lf|\int_\rn K_\eta(x,y)f(y)\,dy\r|
&\le \lf|\int_{\{y\in\rn:\ |x-y|>\eta/2\}} \lf[K_\eta(x,y)-K(x,y)\r]f(y)\,dy\r|\\
&\qquad+\lf|\int_{\{y\in\rn:\ |x-y|>\eta/2\}} K(x,y)f(y)\,dy\r|\\
&  \ls\int_{\{y\in\rn:\ \eta/2<|x-y|\le\eta\}} \frac{|f(y)|}{|x-y|^n}\,dy+T^\ast(f)(x)\\
& \ls \cm(f)(x)+T^\ast(f)(x).
\end{align*}
From this, the H\"older inequality, and Lemma \ref{lem-max-bdd},
it follows that
\begin{align*}
\|L_1\|_q&\le\|b(\cdot)-b(\cdot-\xi)\|_r\lf\|\int_\rn K_\eta(\cdot,y)f(y)\r\|_p\\
&\ls\|b(\cdot)-b(\cdot-\xi)\|_r\lf(\lf\|\cm\r\|_{p\to p}+\lf\|T^\ast\r\|_{p\to p}\r)\|f\|_p.
\end{align*}
By this, the observation that $b$ is uniformly continuous with compact support
[or from the continuity of translations on $L^r(\rn)$],
and Lemma \ref{lem-max-bdd},
we obtain
\begin{align}\label{L1}
\lim_{\xi\to\mathbf{0}}\|L_1\|_q=0.
\end{align}

Now, for any $x,y,\xi\in\rn$ with
$|x-y|<\eta/4$ and $|\xi|<\eta/8$, we have
$|x-y|/\eta<1/2$ and $|x+\xi-y|/\eta<1/2$, which implies that
$\varphi(|x-y|/\eta)=1=\varphi(|x+\xi-y|/\eta)$
and hence
\begin{align}\label{K=0=K}
K_\eta(x,y)=0=K_\eta(x+\xi,y).
\end{align}
Moreover, for any $x,y,\xi\in\rn$ with $|x-y|\ge\eta/4$ and $|\xi|<\eta/8$,
we have $|\xi|\le|x-y|/2$ which, together with \eqref{regular} and \eqref{size},
further implies that
\begin{align}\label{K-K}
&\lf|K_\eta(x,y)-K_\eta(x+\xi,y)\r|\\
&\quad=\lf|K(x,y)\lf[1-\varphi\lf(\frac{|x-y|}{\eta}\r)\r]
-K_\eta(x+\xi,y)\lf[1-\varphi\lf(\frac{|x+\xi-y|}{\eta}\r)\r]\r|\noz\\
&\quad\le\lf|K(x,y)-K(x+\xi,y)\r| \lf|1-\varphi\lf(\frac{|x-y|}{\eta}\r) \r|\noz\\
&\qquad+\lf|K(x+\xi,y)\r|\lf|\varphi\lf(\frac{|x-y|}{\eta}\r)
-\varphi\lf(\frac{|x+\xi-y|}{\eta}\r) \r|\noz\\
&\quad\ls\frac{1}{|x-y|^n}\oz\lf(\frac{|\xi|}{|x-y|}\r)  \noz\\
&\qquad+\frac{\|\varphi'\|_{\fz}}{|x+\xi-y|^n}
\lf|\frac{|x+\xi-y|}{\eta}-\frac{|x-y|}{\eta} \r|
\mathbf1_{\{(x,y)\in\rn\times\rn:\ \frac13\eta\le|x-y|\le2\eta\}}(x,y)\noz\\
&\quad\ls\frac{1}{|x-y|^n}\oz\lf(\frac{|\xi|}{|x-y|}\r)
+\frac{1}{|x-y|^n}\frac{|\xi|}{\eta}
\mathbf1_{\{(x,y)\in\rn\times\rn:\ \frac13\eta\le|x-y|\le2\eta\}}(x,y)\noz\\
&\quad\sim\frac{1}{|x-y|^n}\lf[\oz\lf(\frac{|\xi|}{|x-y|}\r)+\frac{|\xi|}{|x-y|}\r].\noz
\end{align}

By both \eqref{K=0=K} and \eqref{K-K},
we conclude that, for any $x\in\rn$,
\begin{align*}
&\int_{\rn}\lf|K_\eta(x,y)-K_\eta(x+\xi,y)\r| |f(y)|\,dy\\
&\quad\ls
|\xi|\int_\rn\frac{|f(y)|}{|x-y|^{n+1}}\mathbf 1_{\{(x,y)\in\rn\times\rn:\ |x-y|\geq\eta/4\}}\,dy \\
&\qquad+\int_\rn \frac{1}{|x-y|^{n}}\oz\lf(\frac{|\xi|}{|x-y|}\r)|f(y)|
\mathbf 1_{\{(x,y)\in\rn\times\rn:\ |x-y|\geq\eta/4\}}\,dy\\
&\quad\ls|\xi|
\sum_{k=0}^\fz\lf(2^k\eta\r)^{-(n+1)} \int_{\{y\in\rn:\ 2^k\frac\eta4\le|x-y|<2^{k+1}\frac\eta4\}}
|f(y)|\,dy\\
&\qquad+
\sum_{k=0}^\fz\lf(2^k\eta\r)^{-n}
\omega\lf(\frac{|\xi|}{2^{k-2}\eta}\r)
\int_{\{y\in\rn:\ 2^k\frac\eta4\le|x-y|<2^{k+1}\frac\eta4\}}
|f(y)|\,dy\\
&\quad\ls\lf[\frac{|\xi|}{\eta}+\int_0^{8|\xi|/\eta}\frac{\oz(s)}{s}\,ds\r]\cm(f)(x).
\end{align*}
From this, the H\"older inequality, and Lemma \ref{lem-max-bdd},
we deduce that
\begin{align}\label{L3}
&\|L_2\|_q+\|L_3\|_q\\
&\quad\le 2\|b\|_\infty\lf\|\int_{\rn}
\lf|K_\eta(\cdot,y)-K_\eta(\cdot+\xi,y)\r| |f(y)|\mathbf 1_{B(\mathbf 0, R_0)}\,dy\r\|_q\noz\\
&\qquad +\|b\|_r\lf\|\int_{\rn}
\lf|K_\eta(\cdot,y)-K_\eta(\cdot+\xi,y)\r| |f(y)|\,dy\r\|_p\noz\\
&\quad\ls
\lf[\frac{|\xi|}{\eta}+\int_0^{8|\xi|/\eta}\frac{\oz(s)}{s}\,ds\r]
\lf[\|b\|_\infty\|\cm(f\mathbf 1_{B(\mathbf 0, R_0)})\|_q+\|b\|_r\|\cm(f)\|_p\r]\noz\\
&\quad\ls
\lf[\frac{|\xi|}{\eta}+\int_0^{8|\xi|/\eta}\frac{\oz(s)}{s}\,ds\r]
\lf[\|b\|_\infty \|\cm\|_{q\to q}R_0^{n/r} +\|b\|_r\|\cm\|_{p\to p}\r]
\|f\|_p.\noz
\end{align}

Combining \eqref{L1+L2}, \eqref{L1}, \eqref{L3}, \eqref{Dini},
and Lemma \ref{lem-max-bdd},
we conclude that
$$\lim_{\xi\to\mathbf{0}}
\lf\|[b, T_\eta](f)(\cdot+\xi)-[b, T_\eta](f)(\cdot) \r\|_q=0$$
uniformly for any $f\in \cf$,
which implies the condition (iii) of Lemma \ref{lem-FK}.
Thus, $[b,\,T_\eta]$ is a compact operator
for any given $b\in C_{\rm c}^\fz(\rn)$ and $\eta\in(0,\fz)$.
This then finishes the  proof of Theorem \ref{thm-cpt}.
\end{proof}

\begin{remark}
One can avoid to use the compact support of $b$
in \eqref{L3} when $p=q$; see, for instance, \cite{tyyz21PA}.
\end{remark}

Now, we consider a Calder\'on--Zygmund operator $T_\Omega$
with rough homogeneous kernel $\Omega$.
Such type of operators and their commutators have attracted
a lot of attention; see, for instance, \cite{cl22,cg21,cw18}.
Precisely, let $\Omega\in L^1(\mathbb{S}^{n-1})$
be homogeneous of degree zero and have mean value zero,
namely, for any $\mu\in(0,\infty)$ and $x\in\mathbb{S}^{n-1}$,
\begin{equation}\label{homo}
\Omega(\mu x):=\Omega(x)\quad{\rm and}\quad
\int_{\mathbb{S}^{n-1}}\Omega(x)\,d\sigma(x)=0;
\end{equation}
here and hereafter,
$\mathbb{S}^{n-1}:=\{x\in\rn:\ |x|=1\}$ denotes the \emph{unit sphere} of $\rn$
and $d\sigma$ the area measure on $\mathbb{S}^{n-1}$.
Then, for any suitable function $f$ and any $x\in\rn$,
\begin{align*}
T_\Omega(f)(x):=&\,\text{\,p.\,v.}\,\int_{\rn}\frac{\Omega(x-y)}{|x-y|^n}f(y)\,dy\\
:=&\,\lim_{\varepsilon\to0^+}\int_{\varepsilon<|x-y|<1/\varepsilon}
\frac{\Omega(x-y)}{|x-y|^n}f(y)\,dy.
\end{align*}
It is well known that $T_\Omega$ is bounded on $L^p(\rn)$
when $\Omega \in L^{v}(\mathbb{S}^{n-1})$ for some $v\in(1,\fz]$
and, moreover, for any $f\in L^p(\rn)$,
\begin{align}\label{bdd-Toz}
\|T_\Omega\|_p\ls \|\Omega\|_{L^{v}(\mathbb{S}^{n-1})}\|f\|_p
\end{align}
with the implicit positive constant depending only on both $n$ and $v$;
see, for instance, \cite[p.\,79, Theorem 4.2]{d01}.
In particular, if $\Omega\in{\rm Lip\,}(\mathbb{S}^{n-1})$, namely,
$\Omega$ satisfies the \emph{Lipschitz condition}
$$
|\Omega(x)-\Omega(y)|\ls|x-y|\quad {\rm for\ any\ } x,y\in\mathbb{S}^{n-1}
$$
with the implicit positive constant independent of both $x$ and $y$,
then the kernel $K(x,y):=\frac{\Omega(x-y)}{|x-y|^n}$
for any $x,y\in\rn$ with $x\neq y$ satisfies \eqref{size},
\eqref{regular}, and \eqref{Dini},
which implies the following corollary.
\begin{corollary}\label{cpt-Lip}
Let $\Omega\in{\rm Lip\,}(\mathbb{S}^{n-1})$ satisfy \eqref{homo}.
Then Theorem \ref{thm-cpt} holds true for $T_\Omega$.
\end{corollary}
For a rough $\Omega$,
we approximate it via ${\rm Lip\,}(\mathbb{S}^{n-1})$
and hence obtain the following approximation on $[b,T_\Omega]$.
\begin{lemma}\label{lem-Toz}
Let $1<q<p<\fz$, $b\in L^r(\rn)$ with $\frac1r:=\frac1q-\frac1p$,
and $T_\Omega$ be a Calder\'on--Zygmund operator with $\Omega$
satisfying \eqref{homo} and $\Omega \in L^{v}(\mathbb{S}^{n-1})$ for some $v\in(1,\fz)$.
Then, for any $f\in L^p(\rn)$,
$$\lf\|[b,T_\Omega](f) \r\|_q
\le\|b\|_r \lf\|\Omega\r\|_{L^v(\mathbb{S}^{n-1})} \|f\|_p$$
and hence
$$\inf_{\widetilde{\Omega}\in{\rm Lip\,}(\mathbb{S}^{n-1})}
\lf\|[b,T_{\Omega}]-[b,T_{\widetilde{\Omega}}]\r\|_{p\to q}=0.$$
\end{lemma}
\begin{proof}
Let all the symbols be the same as in the present lemma.
Then, by the H\"older inequality and \eqref{bdd-Toz},
we conclude that
\begin{align*}
\lf\|[b,T_\Omega](f) \r\|_q
&\le\lf\|b T_\Omega(f)\r\|_q+\lf\|T_\Omega(bf)\r\|_r
\ls \|b\|_r \lf\|T_\Omega(f)\r\|_p
 +\lf\|\Omega\r\|_{L^v(\mathbb{S}^{n-1})} \|bf\|_q\\
&\ls \|b\|_r \lf\|\Omega\r\|_{L^v(\mathbb{S}^{n-1})} \|f\|_p.
\end{align*}
Thus, for any $\widetilde{\Omega}\in{\rm Lip\,}(\mathbb{S}^{n-1})$,
\begin{align*}
\lf\|[b,T_{\Omega}]-[b,T_{\widetilde{\Omega}}]\r\|_{p\to q}
\ls \|b\|_r \lf\|\Omega-\widetilde{\Omega}\r\|_{L^v(\mathbb{S}^{n-1})},
\end{align*}
which, together with the density of ${\rm Lip\,}(\mathbb{S}^{n-1})$
in $L^v(\mathbb{S}^{n-1})$,
then completes the proof of Lemma \ref{lem-Toz}.
\end{proof}

The second main theorem of this article concerns
the $L^p$-to-$L^q$ compactness for commutators of rough homogeneous
Calder\'on--Zygmund operators.

\begin{theorem}\label{thm=cpt-Toz}
Let $1<q<p<\fz$, $b\in L^r(\rn)$ with $\frac1r=\frac1q-\frac1p$,
and $T_\Omega$ be a Calder\'on--Zygmund operator with rough
homogeneous $\Omega$ satisfying \eqref{homo} and
$\Omega \in L^{v}(\mathbb{S}^{n-1})$ for some $v\in(1,\fz]$.
Then the commutator $[b,T_{\Omega}]$ is compact from $L^p(\rn)$ to $L^q(\rn)$.
\end{theorem}

To prove Theorem \ref{thm=cpt-Toz},
it suffices to use Corollary \ref{cpt-Lip}, Lemma \ref{lem-Toz},
the density argument of compact operators
(see the proof of Theorem \ref{thm-cpt}), and the fact
$L^\fz(\mathbb{S}^{n-1})\subset L^v(\mathbb{S}^{n-1})$
for any $v\in(1,\fz)$; we omit the details.

\begin{remark}\label{b=a+c}
It is fairly obvious to find that both Theorems \ref{thm-cpt} and \ref{thm=cpt-Toz}
still hold true if we replace $b\in L^r(\rn)$ by
$b=a+c$ with $a\in L^r(\rn)$ and $c$ being any constant;
we omit the details.
\end{remark}

As in \cite[Definition 2.1.1]{h21},
$K$ is called a \emph{non-degenerate Calder\'on--Zygmund kernel}
if $K$ satisfies (at least) one of the following two conditions:
\begin{itemize}
\item [{\rm(i)}] $K$ is a Calder\'on--Zygmund kernel satisfying \eqref{size},
\eqref{regular}, $\oz(t)\to0$ as $t\to0^+$ and there exists a $c_0\in(0,\fz)$
such that,
for any $y\in\rn$ and $r\in(0,\fz)$,
there exists an $x\in \rn\setminus B(y,r)$ satisfying
$$|K(x,y)|\ge\frac{1}{c_0 r^n};$$

\item [{\rm(ii)}] $K$ is a homogeneous Calder\'on--Zygmund kernel with
$\Omega\in L^1(\mathbb{S}^{n-1})\setminus\{0\}$.
In particular, there exists a Lebesgue point $\tz_0\in \mathbb{S}^{n-1}$ of $\Omega$
such that $\Omega(\tz_0)\neq 0$.
\end{itemize}
Combining Theorems \ref{thm-cpt} and \ref{thm=cpt-Toz},
Remark \ref{b=a+c}, and \cite[Theorem 1.0.1]{h21},
we immediately obtain the following conclusion on
non-degenerate Calder\'on--Zygmund operators.
Recall that, for any $r\in(0,\fz]$,
\[
\dot L^r(\rn):=\lf\{b\in L^r_\loc(\rn):\ \| b\|_{\dot L^r(\rn)}=\inf_c \| b-c\|_{r}<\infty\r\}.
\]
\begin{theorem}\label{thm-cpt-char}
Let $1<q<p<\fz$ and $T$ be a
non-degenerate Calder\'on--Zygmund operator whose kernel $K$
either
$$\text{satisfies \eqref{size}, \eqref{regular}, and \eqref{Dini}}$$
or
$K(x,y):=\frac{\Omega(x-y)}{|x-y|^n}$ for any $x,y\in\rn$ and $x\neq y$ with
\begin{align*}
\text{$\Omega$ satisfying \eqref{homo} and
$\Omega \in L^{v}(\mathbb{S}^{n-1})$ for some $v\in(1,\fz]$.}
\end{align*}
Then the following three statements are mutually equivalent:
\begin{enumerate}
\item [{\rm(i)}] $[b,T]$ is compact from $L^p(\rn)$ to $L^q(\rn)$;

\item [{\rm(ii)}] $[b,T]$ is bounded from $L^p(\rn)$ to $L^q(\rn)$;

\item [{\rm(iii)}] $b\in \dot L^r(\rn)$ with $\frac1r:=\frac1q-\frac1p$.
\end{enumerate}
\end{theorem}
\begin{proof}
The implication (i) $\Longrightarrow$ (ii) directly follows from the definition
of compact operators,
and the implication (ii) $\Longrightarrow$ (iii) is a part of
\cite[Theorem 1.0.1]{h21}.
Moreover, using Theorems \ref{thm-cpt} and \ref{thm=cpt-Toz} and
Remark \ref{b=a+c},
we obtain the implication (iii) $\Longrightarrow$ (i).
This then finishes the proof of Theorem \ref{thm-cpt-char}.
\end{proof}

\section{Multilinear case}\label{s3}

In this section, we briefly discuss how to
extend the previous results to the multilinear setting.
Let us begin with recalling some definitions and notation.
Throughout this section, we fix an $m\in\nn$ with $m\ge2$.
Let
$$\Delta:=\left\{(x, y_1,\ldots, y_m)\in (\rn)^{m+1}:\ x=y_1=\cdots =y_m\right\}$$
be the diagonal in $(\rn)^{m+1}$.
A function $K:\ (\rn)^{m+1}\setminus \Delta \to \mathbb C$ is called a
\emph{multilinear Calder\'on--Zygmund kernel} if there exists a positive constant
$C$ such that
\begin{align}\label{eq:size}
|K(x, y_1,\ldots, y_m)|\le \frac{C}{(\sum_{i=1}^m|x-y_i|)^{mn}}\
\text{ for any $(x, y_1,\ldots, y_m)\in(\rn)^{m+1}\setminus \Delta$},
\end{align}
and, for any $(x, y_1,\ldots, y_m)\in(\rn)^{m+1}\setminus \Delta$
and $h\in\rr$ with $|h|\le \frac 12 \max_{i\in\{1,\ldots,m\}} |x-y_i|$,
\begin{align}\label{eq:regularity}
&|K(x+h,y_1,\ldots, y_m)-K(x,y_1,\ldots, y_m)|\\
&\qquad+ \sum_{i=1}^{m}\lf|K(x,y_1,\ldots, y_i+h,\ldots,y_m)-K(x,y_1,\ldots,y_i,\ldots,y_m)\r|\noz\\
&\quad\le \frac{C}{(\sum_{i=1}^m|x-y_i|)^{mn}}\omega
\left(\frac{|h|}{\sum_{i=1}^m|x-y_i|}\right)\nonumber,
\end{align}
where $\omega:\ [0,1)\to[0,\fz)$ is an increasing function
with $\omega(0)=0$ and satisfies the Dini condition \eqref{Dini}.
Then $T$ is called an \emph{$m$-linear Calder\'on--Zygmund operator}
if $T$ is initially bounded from
$L^{q_1}(\rn)\times \cdots \times L^{q_m}(\rn)$ to $L^q(\rn)$
with $q_i\in (1,\infty]$ for any $i\in\{1,\ldots,m\}$
and $\frac 1q:=\sum_{i=1}^m \frac 1{q_i}\in(0,\fz)$,
and there exists a multilinear Calder\'on--Zygmund kernel $K$ such that,
for any $f_1,\dots, f_m\in C_{\rm c}^\infty(\rn)$
and any $x\notin \bigcap_{i=1}^m \supp f_i$,
\begin{equation}\label{eq:kernelrep}
 T(f_1,\dots, f_m)(x)=\int_{\mathbb R^{mn}} K(x,y_1,\ldots, y_m)
\prod_{i=1}^m f_i(y_i)\, dy_1\cdots d y_m,
\end{equation}
where, for any $i\in\{1,\ldots,m\}$,
$\supp f_i$ denotes the \emph{support} of $f_i$,
namely, the closure in $\rn$ of the set $\{x\in\rn:\ f_i(x)\neq 0\}$.
Moreover, we say $T$ is a \emph{non-degenerate multilinear Calder\'on--Zygmund operator}
if there exists a function $K$ such that \eqref{eq:size}, \eqref{eq:regularity},
and \eqref{eq:kernelrep} hold true with $\omega(0)\to 0$ when $t\to 0^+$
and, in addition, there exists a positive constant $c_0$ such that,
for any $y\in \rn$ and $r\in(0,\fz)$, there exists an $x\in \rn\setminus B(y, r)$ satisfying
\begin{equation}\label{eq:non-degen}
|K(x,y,\dots,y)|\ge \frac 1{c_0 r^{mn}}.
\end{equation}

For any $i\in\{1,\dots,m\}$,
the \emph{multilinear Calder\'on--Zygmund commutator} $[b,T]_i$
is defined by setting, for any suitable functions $\{f_i\}_{i=1}^m$ and any $x\in\rn$,
\begin{align*}
[b,T]_i(f_1,\dots,f_m)(x):=b(x)T(f_1,\dots,f_m)(x)-T(f_1,\dots,bf_i,\dots,f_m)(x).
\end{align*}
Without loss of generality, we may only consider the case $i=1$,
namely, the commutator $[b,T]_1$.
First of all, we establish the boundedness of
multilinear Calder\'on--Zygmund commutators.
In what follows, for any $\gz,\gz_1,\ldots,\gz_m\in(0,\fz]$,
we use $\|T\|_{(\gz_1,\ldots,\gz_m)\to \gz}$ to denote the \emph{operator norm}
of $T$ from $L^{\gz_1}(\rn)\times \cdots\times L^{\gz_m}(\rn)$ to $L^\gz(\rn)$.
\begin{proposition}\label{prop-bdd}
Let $0<q<p<\fz$, $p_i\in(1,\fz)$ for any $i\in\{1,\ldots,m\}$,
$\frac 1p=\frac 1{p_1}+\cdots+\frac 1{p_m}$,
$b\in \dot L^r(\rn)$ with $\frac1r:=\frac1q-\frac1p<\frac 1{p_1'}$,
and $T$ be an $m$-linear Calder\'on--Zygmund operator,
where $\frac1{p_1}+\frac1{p_1'}=1$.
Then the commutator $[b,T]_1$ is bounded
from $L^{p_1}(\rn)\times \cdots\times L^{p_m}(\rn)$ to $L^q(\rn)$.
\end{proposition}
\begin{proof}
Let all the symbols be the same as in the present proposition.
Then, by the H\"older inequality, the boundedness of
$T$ from $L^{p_1}(\rn)\times \cdots\times L^{p_m}(\rn)$ to $L^p(\rn)$
(see, for instance, \cite{dhl18,l18}), and the observation
$1/(\frac1{p_1}+\frac1r)=\frac {p_1r}{p_1+r}>1$, together with the boundedness of $T$
from $L^{\frac {p_1r}{p_1+r}}(\rn)\times \cdots\times L^{p_m}(\rn)$
to $L^q(\rn)$ (see, for instance, \cite{dhl18,l18} again),
we conclude that,
for any $(f_1,\ldots,f_m)\in L^{p_1}(\rn)\times \cdots\times L^{p_m}(\rn)$,
\begin{align*}
&\lf\|[b,T]_1(f_1,\ldots,f_m)\r\|_q\\
&\quad\le\lf\|bT(f_1,\ldots,f_m)\r\|_q+\lf\|T(bf_1,f_2,\ldots,f_m)\r\|_q\\
&\quad\le\|b\|_r\|T(f_1,\ldots,f_m)\|_p
+\|T\|_{(\frac {p_1r}{p_1+r},p_2,\ldots,p_m)\to q}
\|bf_1\|_{\frac {p_1r}{p_1+r}}\prod_{i=2}^m\|f_i\|_{p_i}\\
&\quad\le\lf[\|T\|_{(p_1,\ldots,p_m)\to p}
+\|T\|_{(\frac {p_1r}{p_1+r},p_2,\ldots,p_m)\to q}\r]
\|b\|_r \prod_{i=1}^m\|f_i\|_{p_i},
\end{align*}
which completes the proof of Proposition \ref{prop-bdd}.
\end{proof}
Next, we consider the compactness of $[b,T]_1$ and likewise introduce
the smooth truncated $m$-linear Calder\'on--Zygmund kernel:
for any $\eta\in(0,\fz)$, $x\in\rn$, and $y_i\in\rn$ with $i\in\{1,\ldots,m\}$,
\begin{align*}
K_\eta(x,y_1,\ldots, y_m):=
K(x,y_1,\ldots, y_m)\lf[1-\varphi\lf(\frac{\max_{1\le i\le m}|x-y_i|}{\eta}\r)\r]
\end{align*}
with the same $\varphi$ as in \eqref{phi-def}.
Let $b\in C_{\rm c}^1(\rn)$ and, for any $i\in\{1,\ldots,m\}$,
$f_i\in L^{p_i}(\rn)$ with $p_i\in(1,\fz)$ and $\frac1p:=\frac1{p_1}+\cdots+\frac1{p_m}$.
We may assume that $\supp(b)\subset B(\mathbf{0},R_0)$
for some positive constant $R_0$.
Then, by similar arguments to the proof of Lemma \ref{lem-Teta} above
(see also \cite[Lemma 2.1]{bdmt15}),
we have, for any $\eta\in(0,\fz)$ and $x\in\rn$,
\begin{align*}
&\lf|[b,T]_1(f_1,\ldots, f_m)(x)- [b,T_\eta]_1(f_1,\ldots, f_m)(x)\r|\\
&\quad\lesssim \eta \|\nabla b\|_\infty
\cm (f_1 {\mathbf 1}_{B(\mathbf{0},R_0+\eta)}, f_2, \ldots, f_m)(x),
\end{align*}
where the implicit positive constant is independent of $\eta$,
$b$, $\{f_i\}_{i=1}^m$, and $x$,
and where $\cm$ stands for the \emph{$m$-linear Hardy--Littlewood maximal operator}
defined by setting, for any $g_i\in L^1_\loc(\rn)$ with
$i\in\{1,\ldots,m\}$ and for any $x\in\rn$,
\[
\mathcal M (g_1, \ldots, g_m)(x):=
\sup_{{\rm ball}\ B\ni x}\prod_{i=1}^m\frac1{|B|}\int_B\lf|g_i(y_i)\r|\,dy_i.
\]
Then this immediately implies the multilinear analogy of Lemma \ref{lem-Teta}.
To obtain the multilinear analogy of Theorem \ref{thm-cpt}, note that,
for any $x\in \mathbb R^n$ with $|x|>M>2R_0$, we have
\begin{align*}
&\lf|[b,T_\eta]_1(f_1,\ldots, f_m)(x)\r| \\
&\quad\lesssim \|b\|_\infty \int_{\mathbb R^{mn}}
\frac{ (|f_1| \mathbf 1_{B(\mathbf 0, R_0)})(y_1)
\prod_{i=2}^m |f_i(y_i)|}{(\sum_{i=1}^m |x-y_i|)^{mn}}\,dy_1\cdots dy_m\\
&\quad\lesssim \|b\|_\infty \int_{\mathbb R^n}
\frac{  (|f_1| \mathbf 1_{B(\mathbf 0, R_0)})(y_1)}{|x-y_1|^n}\,dy_1
\prod_{i=2}^m \mathcal M f_i(x)
\lesssim \frac{\|b\|_\fz\|f_1\|_{p_1}R_0^{n/p_1'}}{|x|^n}\prod_{i=2}^m \mathcal M f_i(x),
\end{align*}
where the implicit positive constant is independent of $\eta$,
$b$, $\{f_i\}_{i=1}^m$, and $x$.
From this, $\frac1q:=\frac1p+\frac1r=\frac1{p_1}+\cdots+\frac1{p_m}+\frac1r$,
and the H\"older inequality, it follows that
\begin{align*}
&\lf\|[b,T_\eta]_1(f_1,\ldots, f_m){\mathbf 1}_{\{x\in\rn:\ |x|>M\}} \r\|_q\\
&\quad\lesssim
\|b\|_\fz R_0^{n/p_1'} \lf\| |x|^{-n} {\mathbf 1}_{\{x\in\rn:\ |x|>M\}} \r\|_{q_1}
\prod_{i=1}^m \|f_i\|_{p_i}\\
&\quad\sim \|b\|_\fz R_0^{n/p_1'} M^{-\frac n{q_1'}} \prod_{i=1}^m \|f_i\|_{p_i},
\end{align*}where $
\frac 1{q_1}:=\frac 1{p_1}+\frac 1r<1$ and
the implicit positive constant is independent of $\eta$,
$b$, $\{f_i\}_{i=1}^m$, and $x$.
This verifies the condition (ii) of Lemma \ref{lem-FK}.
The remaining arguments are quite similar to the proof of Theorem \ref{thm-cpt}.
So, by omitting some details, we obtain the following result.

\begin{theorem}\label{thm-multi}
Let $0<q<p<\fz$, $p_i\in(1,\fz)$ for any $i\in\{1,\ldots,m\}$,
$\frac 1p=\frac 1{p_1}+\cdots+\frac 1{p_m}$,
$b\in \dot L^r(\rn)$ with $\frac1r:=\frac1q-\frac1p<\frac 1{p_1'}$,
and $T$ be an $m$-linear Calder\'on--Zygmund operator,
where $\frac1{p_1}+\frac1{p_1'}=1$.
Then the commutator $[b,T]_1$ is compact
from $L^{p_1}(\rn)\times \cdots\times L^{p_m}(\rn)$ to $L^q(\rn)$.
\end{theorem}

One may be curious about whether or not the opposite direction holds true.
Namely, if we know that $T$ is an $m$-linear Calder\'on--Zygmund operator
associated with some non-degenerate kernel
and that the commutator $[b,T]_1$ is compact
from $L^{p_1}(\rn)\times \cdots\times L^{p_m}(\rn)$ to $L^q(\rn)$,
can we show that $b=a+c$ for some $a\in L^r$ and some constant $c$?
We give an affirmative answer to this question when $q\in(1,\fz)$
and $b$ is real-valued as follows.

\begin{theorem}\label{thm:multi}
Let $1< q<p<\fz$, $p_i\in(1,\fz)$ for any $i\in\{1,\ldots,m\}$,
and $\frac 1p=\frac 1{p_1}+\cdots+\frac 1{p_m}$.
Let $T$ be an $m$-linear non-degenerate  Calder\'on--Zygmund operator
and $b\in L_{\loc}^1(\rn)$ be real-valued.
If the commutator $[b,T]_1$ is bounded
from $L^{p_1}(\rn)\times \cdots\times L^{p_m}(\rn)$ to $L^q(\rn)$,
then $b\in \dot L^r(\rn)$ with $\frac1r:=\frac1q-\frac1p$.
\end{theorem}
To prove the above theorem,
we need the following characterization of $\dot L^r(\rn)$,
whose proof can be found at \cite[Proposition 3.2]{AHLMO}.
In what follows, a collection $\mathscr{S}$ of cubes is said to be \emph{sparse}
if there exists a family of pairwise disjoint subsets,
$\{E(S)\}_{S\in\mathscr{S}}$, such that, for any $S\in\mathscr{S}$,
$$E(S)\subset S\text{ and }|E(S)| \ge \frac12 |S|;$$
for any locally integrable function $b$
and any bounded measurable set $S$ with $|S|>0$, let
$$\langle b \rangle_S:=\frac{1}{|S|}\int_S b(x)\,dx.$$

\begin{proposition}\label{prop:lr}
Let $r\in (1,\infty)$ and $b\in L^r_\loc(\rn)$. Then
\begin{align*}
\| b\|_{\dot L^r(\rn)}\sim
\sup \lf\{ \sum_{S\in \mathscr S}\lambda_S
\int_S |b(x)- \langle b\rangle_S|\,dx:\ \mathscr S \text{ is sparse, }
\sum_{S\in \mathscr S} |S|\lambda_S^{r'}\le 1 \r\}
\end{align*}
with the positive equivalence constants independent of $b$.
\end{proposition}

Now, we are ready to prove Theorem \ref{thm:multi},
via following the median method used in \cite{h21,Li-22}.
In what follows, for any locally integrable functions $f$ and $g$, let
$\langle f,g \rangle:=\int_{\rn}|f(x)g(x)|\,dx$.
\begin{proof}[Proof of Theorem \ref{thm:multi}]
Let all the symbols be the same as in the present theorem.
By \eqref{eq:non-degen}, we find that, for any cube $Q$,
there exists a cube $\widetilde Q$ such that $\ell(Q)=\ell(\widetilde Q)$
and $\dist(Q, \widetilde Q)\ge C_0\ell(Q)$,
and there exists a $\sigma_Q \in \mathbb C$ with $|\sigma_Q|=1$ and
\begin{align}\label{real>}
\Re \lf(\sigma_Q K(x, y_1,\dots, y_m)\r)
\sim |Q|^{-m},\quad
\forall\, x\in \widetilde Q,\ \forall \, y_1,\dots, y_m\in Q;
\end{align}
see, for instance, \cite[Proposition 2.2.1 and Remark 4.1.2]{h21}.
Here and hereafter, we use $\Re(z)$ to denote the \emph{real part} of any $z\in\cc$.
Now, we apply Proposition \ref{prop:lr} and fix a sparse family $\mathscr S$.
For any cube $S\in \mathscr S$, let $\alpha_S$ be the \emph{median} of $b$ on $\widetilde S$,
namely,
\[
\min\lf( |\widetilde S\cap \{b\le \alpha_S\}|,\ |\widetilde S\cap\{b\ge \alpha_S\} |\r)
\ge \frac 12|\widetilde S|=\frac 12 |S|.
\]
By this and \eqref{real>}, we obtain,
for any $x\in \widetilde S \cap\{b\ge \alpha_S\}$,
\begin{align*}
&\int_S \lf[b(y_1)-\alpha_S\r]_-\,dy_1\\
&\quad= \int_{S\cap \{b\le \alpha_S\}} \lf[\alpha_S-b(y_1)\r]\,dy_1\\
&\quad\lesssim |S|\Re\lf( \sigma_S
\int_{S^{m-1}}\int_{S\cap \{y_1: b(y_1)\le \alpha_S\}}
[b(x)-b(y_1)]K(x, y_1,\cdots, y_m)\,dy_1 \cdots dy_m\r)\\
&\quad=|S| \Re\lf( \sigma_S [b, T]_1 (\mathbf 1_{S\cap \{b\le \alpha_S\}},
\mathbf 1_S,\cdots, \mathbf 1_S)(x)\r).
\end{align*}
Likewise, for any $x\in \widetilde S\cap\{b\le \alpha_S\}$,
we have
\begin{align*}
&\int_S \lf[b(y_1)-\alpha_S\r]_+\,dy_1
\lesssim -|S| \Re\lf( \sigma_S [b, T]_1 (\mathbf 1_{S\cap \{b\ge \alpha_S\}},
\mathbf 1_S,\cdots, \mathbf 1_S)(x)\r).
\end{align*}
Thus, for any $\{\lambda_S\}_{S\in\mathscr{S}}\subset[0,\fz)$, we find that
\begin{align*}
&\sum_{S\in \mathscr S}\lambda_S \int_S |b(x)- \langle b\rangle_S|\,dx\\
&\quad \le 2\sum_{S\in \mathscr S}\lambda_S\int_S |b(x)-\alpha_S|\,dx\\
&\quad \ls\sum_{S\in \mathscr S}\lambda_S |S|
\Re\lf( \sigma_S \lf\langle [b, T]_1 (\mathbf 1_{S\cap \{b\le \alpha_S\}},
\mathbf 1_S,\cdots, \mathbf 1_S)\r\rangle_{\widetilde S \cap\{b\ge \alpha_S\}}\r)\\
&\qquad-\sum_{S\in \mathscr S}\lambda_S |S|
\Re\lf( \sigma_S \lf\langle [b, T]_1 (\mathbf 1_{S\cap \{b\ge \alpha_S\}},
\mathbf 1_S,\cdots, \mathbf 1_S)\r\rangle_{\widetilde S \cap\{b\le \alpha_S\}}\r)\\
&\quad \lesssim \sum_{S\in \mathscr S}\lambda_S |S|  \lf\langle
\lf| [b, T]_1 (\mathbf 1_{S\cap \{b\le \alpha_S\}}, \mathbf 1_S,\cdots,
\mathbf 1_S)\r|\r\rangle_{\widetilde S}\\
&\qquad+\sum_{S\in \mathscr S}\lambda_S |S|  \lf\langle \lf| [b, T]_1
(\mathbf 1_{S\cap \{b\ge \alpha_S\}}, \mathbf 1_S,\cdots, \mathbf 1_S)\r|\r\rangle_{\widetilde S} .
\end{align*}
We next only focus on estimating the first term in the last step
since the other one is similar.
Moreover, by the monotone convergence theorem,
we may assume that $\mathscr S$ contains only finitely many elements.
Now, let $g_S$ be the function such that
\[
\lf| [b, T]_1 (\mathbf 1_{S\cap \{b\le \alpha_S\}}, \mathbf 1_S,\cdots, \mathbf 1_S)\r|
=  [b, T]_1 (\mathbf 1_{S\cap \{b\le \alpha_S\}}, \mathbf 1_S,\cdots, \mathbf 1_S) g_S.
\]
For each $j\in\{1,\ldots, m\}$, let $\{\ve_{S}^{(j)}\}_{S\in \mathscr S}$
be a collection of independent random signs,
and we denote by $\mathbb E^{(j)}$ the corresponding \emph{expectation},
namely, for any function $f$ defined on two random signs,
$$\mathbb E^{(j)}f(\ve_1,\ve_2)
:=\sum_{\ve_1=\pm1,\ \ve_2=\pm1}
\frac{f(\ve_1,\ve_2)}{4},$$
where $\ve_1,\ve_2\in\{\ve_{S}^{(j)}\}_{S\in \mathscr S}$.
Write $\mathbb E:= \mathbb E^{(1)}\cdots \mathbb E^{(m)}$ and observe that
$$1=r'\lf(\frac1{q'}+\frac1p\r)
=\frac{r'}{q'}+\frac{r'}{p_1}+\cdots+\frac{r'}{p_m},$$
where $\frac1r+\frac1{r'}=1=\frac1q+\frac1{q'}$.
From this, the linearity of $[b, T]_1(\cdot)$, the H\"older inequality,
Proposition \ref{prop-bdd},
\cite[Lemma 2.5.4]{h21}, and $\sum_{S\in \mathscr S} |S|\lambda_S^{r'}\le 1$,
it follows that
\begin{align*}
&\sum_{S\in \mathscr S}\lambda_S |S|  \lf\langle \lf| [b, T]_1
\lf(\mathbf 1_{S\cap \{b\le \alpha_S\}}, \mathbf 1_S,\ldots, \mathbf 1_S\r)\r|\r\rangle_{\widetilde S} \\
&\quad=\sum_{S\in \mathscr S} \lf\langle[b, T]_1\left(\lambda_S^{\frac{r'}{p_1}}\mathbf 1_{S\cap \{b\le \alpha_S\}}, \lambda_S^{\frac{r'}{p_2}}\mathbf 1_S,\ldots, \lambda_S^{\frac{r'}{p_m}}\mathbf 1_S\right),\  \lambda_S^{\frac{r'}{q'}}g_S  \mathbf 1_{\widetilde S}\r\rangle \\
&\quad= \mathbb E\lf\langle[b, T]_1\lf(\sum_{S_1\in \mathscr S}\ve_{S_1}^{(1)}
\lambda_{S_1}^{\frac{r'}{p_1}}\mathbf 1_{S_1\cap \{b\le \alpha_{S_1}\}},
\sum_{S_2\in \mathscr S}\ve_{S_2}^{(1)} \ve_{S_2}^{(2)}\lambda_{S_2}^{\frac{r'}{p_2}}\mathbf 1_{S_2},
\ldots,\r.\r.\\
&\qquad\lf.\lf.\sum_{S_m\in \mathscr S} \ve_{S_{m}}^{(m-1)}\ve_{S_m}^{(m)}\lambda_{S_m}^{\frac{r'}{p_m}}\mathbf 1_{S_m}\r),\sum_{S_{m+1}\in \mathscr S}\ve_{S_{m+1}}^{(m)}
\lambda_{S_{m+1}}^{\frac{r'}{q'}}g_S  \mathbf 1_{\widetilde S_{m+1}}\r\rangle\\
&\quad\lesssim \mathbb E\lf\|\sum_{S_1\in \mathscr S}\ve_{S_1}^{(1)}
\lambda_{S_1}^{\frac{r'}{p_1}}\mathbf 1_{S_1\cap \{b\le \alpha_{S_1}\}}\r\|_{p_1}
\cdots
\lf\|\sum_{S_m\in \mathscr S} \ve_{S_{m}}^{(m-1)}\ve_{S_m}^{(m)}
\lambda_{S_m}^{\frac{r'}{p_m}}\mathbf 1_{S_m}\r\|_{p_m}\\
&\qquad\cdot \lf\|\sum_{S_{m+1}\in \mathscr S}\ve_{S_{m+1}}^{(m)}\lambda_{S_{m+1}}^{\frac{r'}{q'}}g_S
\mathbf 1_{\widetilde S_{m+1}}\r\|_{q'}\\
&\quad\lesssim  \lf\|\sum_{S_1\in \mathscr S}
\lambda_{S_1}^{\frac{r'}{p_1}}\mathbf 1_{S_1}\r\|_{p_1}
\cdots
\lf\|\sum_{S_m\in \mathscr S}
\lambda_{S_m}^{\frac{r'}{p_m}}\mathbf 1_{S_m}\r\|_{p_m}
\cdot \lf\|\sum_{S_{m+1}\in \mathscr S} \lambda_{S_{m+1}}^{\frac{r'}{q'}}
\mathbf 1_{\widetilde S_{m+1}}\r\|_{q'}\\
&\quad\sim\lf(\sum_{S\in \mathscr S}\lambda_S^{r'}|S|\r)
^{\frac1{p_1}+\cdots+\frac1{p_m}+\frac1{q'}}
\lesssim 1.
\end{align*}
This, together with Proposition \ref{prop:lr},
then completes the proof of Theorem \ref{thm:multi}.
\end{proof}

The above multilinear results can also be extended to the iterated case.
In what follows, for any $k\in\nn$, let $T^{(k)}_{b,1}:=[b,T^{(k-1)}_{b,1}]_1$
be the $k$-times iterated commutator of $[b,T]_1=:T^{(1)}_{b,1}$.
Combining \cite[Theorem 4.0.1]{h21} and Theorems \ref{thm-multi} and \ref{thm:multi},
we immediately have the following results; we omit the details here.
\begin{theorem}\label{thm-multi-iterated}
Let $0<q<p<\fz$, $p_i\in(1,\fz)$ for any $i\in\{1,\ldots,m\}$,
and $\frac 1p=\frac 1{p_1}+\cdots+\frac 1{p_m}$.
Let $k\in\nn$ and $b\in \dot L^{kr}(\rn)$ with $\frac1r:=\frac1q-\frac1p<\frac 1{p_1'}$,
where $\frac1{p_1}+\frac1{p_1'}=1$.
Let $T$ be an $m$-linear Calder\'on--Zygmund operator.
Then the commutator $T^{(k)}_{b,1}$ is compact from $L^{p_1}(\rn)\times \cdots\times L^{p_m}(\rn)$ to $L^q(\rn)$.
\end{theorem}
\begin{theorem}\label{thm:multi-iterated}
Let $1< q<p<\fz$, $p_i\in(1,\fz)$ for any $i\in\{1,\ldots,m\}$,
$\frac 1p=\frac 1{p_1}+\cdots+\frac 1{p_m}$, and $k\in\nn$.
Let $T$ be an $m$-linear non-degenerate Calder\'on--Zygmund operator
and $b\in L_{\loc}^k(\rn)$ be real-valued.
If the commutator $T^{(k)}_{b,1}$ is bounded
from $L^{p_1}(\rn)\times \cdots\times L^{p_m}(\rn)$ to $L^q(\rn)$,
then $b\in \dot L^{kr}(\rn)$ with $\frac1r:=\frac1q-\frac1p$.
\end{theorem}

\begin{proof}
We sketch the proof of the present theorem.
Let all the symbols be the same as in the present theorem.
First of all, the same arguments as that used in the proof of Theorem \ref{thm:multi}
give us that
\begin{align*}
\int_S |b(x)-\langle b\rangle_S|^k\,dx
&\lesssim   |S|  \lf\langle \lf| T^{(k)}_{b,1}
(\mathbf 1_{S\cap \{b\le \alpha_S\}},
\mathbf 1_S,\cdots, \mathbf 1_S)\r|\r\rangle_{\widetilde S}\\
&\quad+ |S|  \lf\langle \lf| T^{(k)}_{b,1} (\mathbf 1_{S\cap \{b\ge \alpha_S\}},
\mathbf 1_S,\cdots, \mathbf 1_S)\r|\r\rangle_{\widetilde S}.
\end{align*}
Then, for any $\{\lambda_S\}_{S\in \mathscr S}\subset[0,\fz)$ with
$\sum_{S\in \mathscr S} |S|\lambda_S^{(kr)'}\le 1$,
by both the H\"older inequality and
the Riesz representation theorem of $\ell^r$, we have
\begin{align*}
&\sum_{S\in \mathscr S}\lambda_S \int_S |b(x)- \langle b\rangle_S|\,dx\\
&\quad\le \sum_{S\in \mathscr S}\lambda_S
\lf[\int_S |b(x)-\langle b\rangle_S|^k\,dx\r]^{1/k}|S|^{1/{k'}}\\
&\quad\lesssim \sum_{S\in \mathscr S}\lambda_S |S|
\lf\langle \lf| T^{(k)}_{b,1} (\mathbf 1_{S\cap \{b\le \alpha_S\}},
\mathbf 1_S,\cdots, \mathbf 1_S)\r|\r\rangle_{\widetilde S}^{1/k}\\
&\qquad+ \sum_{S\in \mathscr S}\lambda_S |S|
\lf\langle \lf| T^{(k)}_{b,1} (\mathbf 1_{S\cap \{b\ge \alpha_S\}},
\mathbf 1_S,\cdots, \mathbf 1_S)\r|\r\rangle_{\widetilde S}^{1/k}\\
&\quad\lesssim \lf[\sum_{S\in \mathscr S}  |S|
\lf\langle \lf| T^{(k)}_{b,1} (\mathbf 1_{S\cap \{b\le \alpha_S\}},
\mathbf 1_S,\cdots, \mathbf 1_S)\r|\r\rangle_{\widetilde S}^{r}\r]^{1/{(kr)}}
\lf[\sum_{S\in \mathscr S} |S|\lambda_S^{(kr)'}\r]^\frac1{(kr)'}\\
&\qquad+ \lf[\sum_{S\in \mathscr S}  |S|
\lf\langle \lf| T^{(k)}_{b,1} (\mathbf 1_{S\cap \{b\ge \alpha_S\}},
\mathbf 1_S,\cdots, \mathbf 1_S)\r|\r\rangle_{\widetilde S}^{r}\r]^{1/{(kr)}}
\lf[\sum_{S\in \mathscr S} |S|\lambda_S^{(kr)'}\r]^\frac1{(kr)'}\\
&\quad\ls\lf\|\lf\{|S|^\frac1r
\lf\langle \lf| T^{(k)}_{b,1} (\mathbf 1_{S\cap \{b\le \alpha_S\}},
\mathbf 1_S,\cdots, \mathbf 1_S)\r|\r
\rangle_{\widetilde S}\r\}_{S\in\mathscr S}\r\|_{\ell^r}^{1/k}\\
&\qquad+\lf\|\lf\{|S|^\frac1r
\lf\langle \lf| T^{(k)}_{b,1} (\mathbf 1_{S\cap \{b\ge \alpha_S\}},
\mathbf 1_S,\cdots, \mathbf 1_S)\r|\r
\rangle_{\widetilde S}\r\}_{S\in\mathscr S}\r\|_{\ell^r}^{1/k}\\
&\quad\sim \lf[\sum_{S\in \mathscr S} \tau_S |S|
\lf\langle \lf| T^{(k)}_{b,1} (\mathbf 1_{S\cap \{b\le \alpha_S\}},
\mathbf 1_S,\cdots, \mathbf 1_S)\r|\r\rangle_{\widetilde S}\r]^{1/{k}}\\
&\qquad+ \lf[\sum_{S\in \mathscr S} \tau_S |S|
\lf\langle \lf| T^{(k)}_{b,1} (\mathbf 1_{S\cap \{b\ge \alpha_S\}},
\mathbf 1_S,\cdots, \mathbf 1_S)\r|\r\rangle_{\widetilde S}\r]^{1/{k}},
\end{align*}
where the non-negative sequence $\{\tau_S\}_{S\in \mathscr S}$ satisfies
$$\lf\|\lf\{\tau_S |S|^{\frac1{r'}}\r\}_{S\in \mathscr S}\r\|_{\ell^{r'}}
:=\sum_{S\in \mathscr S}\tau_S^{r'}|S|\le 1.$$
Then the remaining arguments are the same as
those used in the proof of Theorem \ref{thm:multi};
we omit the details.
This finishes the proof of Theorem \ref{thm:multi-iterated}.
\end{proof}

\begin{corollary}\label{coro-mk}
Let $1< q<p<\fz$, $p_i\in(1,\fz)$ for any $i\in\{1,\ldots,m\}$,
$\frac 1p=\frac 1{p_1}+\cdots+\frac 1{p_m}$,
and $\frac1r:=\frac1q-\frac1p$.
%<\frac 1{p_1'}$, where $\frac 1{p_1}+\frac 1{p_1'}=1$.
Let $T$ be an $m$-linear Calder\'on--Zygmund operator,
$b\in L_{\loc}^k(\rn)$ be real-valued, and $k\in\nn$.
Then the following three statements are mutually equivalent:
\begin{enumerate}
\item [{\rm(i)}] $T^{(k)}_{b,1}$ is compact from
$L^{p_1}(\rn)\times \cdots\times L^{p_m}(\rn)$ to $L^q(\rn)$;

\item [{\rm(ii)}] $T^{(k)}_{b,1}$ is bounded from
$L^{p_1}(\rn)\times \cdots\times L^{p_m}(\rn)$ to $L^q(\rn)$;

\item [{\rm(iii)}] $b\in \dot L^{kr}(\rn)$.
\end{enumerate}
\end{corollary}

\begin{proof}
The implication (i) $\Longrightarrow$ (ii) directly follows from the definition
of compact operators.
Moreover, using Theorem \ref{thm:multi-iterated},
we obtain the implication (ii) $\Longrightarrow$ (iii).
Furthermore, Theorem \ref{thm-multi-iterated} shows
the implication (iii) $\Longrightarrow$ (i),
which completes the proof of Corollary \ref{coro-mk}. 
We note that the assumption of Theorem \ref{thm-multi-iterated} 
that $\frac1r:=\frac1q-\frac1p<\frac 1{p_1'}$ follows from the present 
assumptions; namely, we have $\frac1q<1=\frac{1}{p_1}+\frac{1}{p_1'}<\frac{1}{p}+\frac{1}{p_1'}$.
This finishes the proof of Corollary \ref{coro-mk}.
\end{proof}

\begin{remark}
\begin{enumerate}
\item [{\rm(i)}]
By some routine modifications, we find that
Proposition \ref{prop-bdd} and Theorems \ref{thm-multi}, \ref{thm:multi},
\ref{thm-multi-iterated}, and \ref{thm:multi-iterated}
hold true with $[b,T]_1$ replaced by $[b,T]_i$ for any $i\in\{2,\ldots,m\}$.

\item [{\rm(ii)}]
It is still \emph{unclear} whether or not the
assumptions that $q>1$ and that $b$ is real-valued
in Theorems \ref{thm:multi} and \ref{thm:multi-iterated} are necessary.
\end{enumerate}
\end{remark}

%%%%%%%%%%%%%%%%%%%%%%%%%%%%%%%%%%%%%%%%%%%%%%%%%%%%%%%%%%%%%%%%%%%%%%%%%%%%%%%%%%

%%%%%%%%%%%%%%%%%%%%%%%%%%%%%%%%%%%%%%%%%%%%%%%%%%%%%%%%%%%%%%%%%%%%%%%%%%%%%%%%%%

\bigskip

\noindent Tuomas Hyt\"onen

\medskip

\noindent Department of Mathematics and Statistics,
University of Helsinki, P.O. Box 68 (Pietari Kalmin katu 5), 00014 Helsinki, Finland

\smallskip

\noindent {\it E-mail}: \texttt{tuomas.hytonen@helsinki.fi}

\bigskip

\noindent Kangwei Li

\medskip

\noindent Center for Applied Mathematics, Tianjin University, Tianjin 300072, The People's Republic of China

\smallskip

\noindent {\it E-mail}: \texttt{kli@tju.edu.cn}

\bigskip

\noindent Jin Tao

\medskip

\noindent Hubei Key Laboratory of Applied Mathematics,
Faculty of Mathematics and Statistics,
Hubei University, Wuhan 430062, The People's Republic of China

\smallskip

\noindent {\it E-mail}: \texttt{jintao@hubu.edu.cn}

\bigskip

\noindent Dachun Yang (Corresponding author)

\medskip

\noindent Laboratory of Mathematics and Complex Systems
(Ministry of Education of China),
School of Mathematical Sciences, Beijing Normal University,
Beijing 100875, The People's Republic of China

\smallskip

\noindent{\it E-mail}: \texttt{dcyang@bnu.edu.cn}

\end{document}